\documentclass[11pt,reqno,oneside]{amsart}
\usepackage[T2A]{fontenc}
\usepackage[utf8]{inputenc}
\usepackage[english]{babel}
\usepackage[margin=2.75cm]{geometry}

\usepackage{amsmath, amsfonts}
\usepackage{bm} 
\usepackage{mathtools}
\providecommand\given{} 
\newcommand\SetSymbol[1][]{%
	\nonscript\:#1\vert
	\allowbreak \nonscript\:	\mathopen{}}
\DeclarePairedDelimiterX\Set[1]\{\}{%
	\renewcommand\given{\SetSymbol[\delimsize]}	#1} 

\usepackage{array} 
\usepackage{multirow}
\usepackage{rotating}
\usepackage{booktabs}
\usepackage{subcaption}

\usepackage{tikz}
\usetikzlibrary{arrows,patterns,calc}
\usetikzlibrary{decorations.pathreplacing} 

\usepackage[pdfborder={0 0 0}, pdffitwindow = true, pdfstartview = FitH,  pdftitle = {All 2-neighborly d-polytopes with at most d+9 facets},
pdfauthor = {A.N. Maksimenko}
]{hyperref}

\usepackage{numprint} 

\usepackage[lined,ruled]{algorithm2e} 

\theoremstyle{plain}
\newtheorem{theorem}{Theorem}
\newtheorem{proposition}[theorem]{Proposition}
\newtheorem{lemma}[theorem]{Lemma}
\newtheorem{corollary}[theorem]{Corollary}

\theoremstyle{definition}

\theoremstyle{remark}

\newcommand{\Z}{\mathbb{Z}}  
\newcommand{\N}{\mathbb{N}}  
\DeclareMathOperator{\conv}{conv} 
\DeclareMathOperator{\vertices}{vrt} 
\DeclareMathOperator{\facets}{fct} 
\DeclareMathOperator{\rows}{rows} 
\DeclareMathOperator{\cols}{cols} 

\title{All 2-neighborly d-polytopes with at most d + 9 facets}

\author{Aleksandr N. Maksimenko}
\thanks{The study was carried out as part of an internship program for employees of Russian educational and scientific organizations at the Higher School of Economics}
\address{National Research University Higher School of Economics}
\email{maximenko.a.n@gmail.com}
\author{Dmitry V. Gribanov}
\address{National Research University Higher School of Economics}
\email{dimitry.gribanov@gmail.com}
\author{Dmitry S. Malyshev}
\address{National Research University Higher School of Economics}
\email{dsmalyshev@rambler.ru.com}

\begin{document}

\begin{abstract}
We give a~complete enumeration of~all 2-neighborly $d$-polytopes with $d+9$ and less facets.
All of them are realized as 0/1-polytopes, except a 6-polytope $P_{6,10,15}$ with 10 vertices and 15 facets, and pyramids over $P_{6,10,15}$.
In particular, we update the lower bounds for the number of facets of a 2-neighborly $d$-polytope $P$ and showed that the number of facets of $P$ is not less than the number of its vertices $f_0(P)$ for $f_0(P) \le d+10$.
\end{abstract}

\maketitle

\section{Introduction}

Throughout the paper, we consider $d$-dimensional convex polytopes and call them \emph{$d$-polytopes}.
The basic properties of convex polytopes and some notions used below, can be seen in~\cite{Grunbaum:2003}.
A~$d$-polytope $P$ is \emph{2-neighborly}, if any two of its vertices form an edge (1-face) of $P$.
In~particular, simplices are examples of 2-neighborly polytopes.
It~is well known, that there are no other examples in dimension 3, but already in dimension 4, there are infinitely many combinatorial types of 2-neighborly polytopes. 
In the following, we consider only 2-neighborly $d$-polytopes for $d \ge 4$.
Since every face of a 2-neighborly polytope $P$ is 2-neighborly, then $P$ is 3-simplicial (every 3-face is a simplex). 
In~\cite{Maksimenko:2010}, it was conjectured that the number of facets of a 2-neighborly $d$-polytope $P$ is not less than the number of vertices: $\facets(P) \ge \vertices(P)$. The conjecture was proved for the cases $d \le 6$ and $\vertices(P) \le d+5$.

We enumerate all combinatorial types of 2-neighborly $d$-polytopes with at most $d+9$ facets.
They are 11 polytopes, listed in Fig.~\ref{fig:0}, and $k$-fold pyramids over them, $k \in \N$.
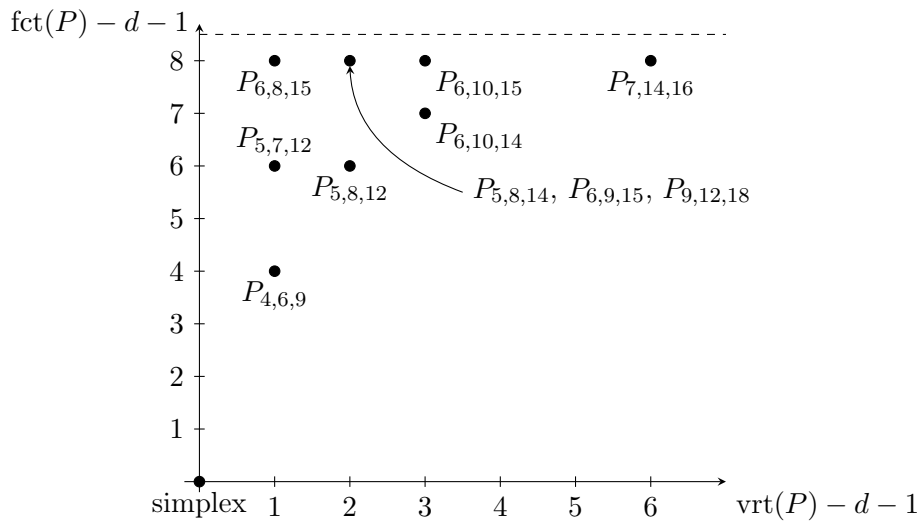
\begin{figure}
	\centering
	\newcommand{\smallradius}{.1}
	\newcommand{\point}{ellipse (0.07 and 0.1) }
	\begin{tikzpicture}[yscale=0.7, fill=black, draw=black, >=stealth]
	\draw [ ->] ( -0.2, 0) -- (7, 0) node[below right] {$\vertices(P) - d - 1$};
	\draw [ ->] (0, -0.2) -- (0, 8.7) node[left] {$\facets(P) - d - 1$};	
	\foreach \x in {1,...,6} \draw (\x, 0.1) -- (\x, -0.1) node[below] {$\x$};
	\foreach \y in {1,...,8} \draw (0.07, \y) -- (-0.07, \y) node[left] {$\y$};
	\draw[dashed] (0, 8.5) -- (7, 8.5);
	\filldraw (0,0) \point node[below] {simplex};
	\filldraw (1,4) \point node[below] {$P_{4,6,9}$};
	\filldraw (1,6) \point node[above] {$P_{5,7,12}$};
	\filldraw (1,8) \point node[below] {$P_{6,8,15}$};
	\filldraw (2,6) \point node[below] {$P_{5,8,12}$};
	\filldraw (2,8) \point;
	\draw[<-] (2,7.9) to[bend right] (3.5, 5.5) node[right] {$P_{5,8,14}$, $P_{6,9,15}$, $P_{9,12,18}$};
	\filldraw (3,7) \point node[below right] {$P_{6,10,14}$};
	\filldraw (3,8) \point node[below right] {$P_{6,10,15}$};
	\filldraw (6,8) \point node[below] {$P_{7,14,16}$};
	\end{tikzpicture}
	\caption{The 2-neighborly $d$-polytopes with at most $d+9$ facets}\label{fig:0}
\end{figure}
All of them, except $P_{6,10,15}$, are realized as 0/1-polytopes (see Fig.~\ref{fig:01}, \ref{fig:3}, and \ref{fig:d6-7}). $P_{6,10,15}$ is realized as a~$\{-1,0,1\}$-polytope (see Fig.~\ref{fig:P61015}).
The polytope $P_{7,14,16}$ was first described in~\cite{Maksimenko:2019}, all other (except $P_{6,10,15}$)~--- in~\cite{Aichholzer:2000}.
Consequently, the conjecture $\facets(P) \ge \vertices(P)$ is true for any 2-neighborly $d$-polytope $P$ with at most $d + 10$ vertices. Also, our result means that we know values of the function
\[
\mu(d, n) = \min \Set*{\facets(P) \given \text{$P$ is a 2-neighborly $d$-polytope with $n$ vertices}}
\]
for $n \in \{d+1, d+2, d+3, d+4, d+7\}$, and $\mu(d,n) \ge d+10$ for $n \ge d+5$, $n \ne d+7$.

In Section~\ref{sec:Gale}, we consider the case $\vertices(P) \le d+3$, by using Gale diagrams.
In Section~\ref{sec:IncMat}, we investigate properties of facet-vertex incidence matrices of 2-neighborly polytopes and describe an~algorithm for enumerating such matrices with small numbers of rows and columns.
In Section~\ref{sec:d+4}, with help of the computer program, we prove that there are no 2-neighborly $d$-polytopes with at least $d+4$ vertices and at most $d+9$ facets, except $P_{6,10,14}$, $P_{6,10,15}$, $P_{7,14,16}$, and $k$-fold pyramids over them.
The computer program and the descriptions of polytopes are available at \url{https://github.com/maksimenko-a-n/2neighborly-inc-matrices}.

\section{d-Polytopes with at most d + 3 vertices}
\label{sec:Gale}

It is well known~\cite[Sec.~6.3]{Grunbaum:2003} that the combinatorial type of a $d$-polytope $P$ with $d+2$ vertices can be coded by a Gale diagram $(m_0, \{m_1, m_{-1}\})$, where $m_0, m_1, m_{-1} \in \Z$, $m_0 \ge 0$, $m_1 \ge 2$, $m_{-1} \ge 2$, and $m_0 + m_1 + m_{-1} = d+2$.
A polytope is 2-neighborly iff $m_1 \ge 3$, $m_{-1} \ge 3$.
The number of facets is equal to $m_0 + m_{-1} m_{1}$.
Hence, the inequality $\facets(P) \le d+9$ is equivalent to $m_0 + m_{-1} m_{1} \le m_0 + m_1 + m_{-1} + 7$.
Thus, pair $\{m_{-1}, m_{1}\}$ must be one of the following types: $\{3,3\}$, $\{3,4\}$, $\{3,5\}$.
For $m_0 = 0$, these three types can be represented by 0/1-polytopes (see Fig~\ref{fig:01}).
A polytope with $m_0 > 0$ is an $m_0$-fold pyramid over a polytope with $m_0 = 0$.
Therefore, every Gale diagram of types $(m_0, \{3,3\})$, $(m_0, \{3,4\})$, $(m_0, \{3,5\})$ can be represented by a 0/1-polytope.

\begin{figure}
	\begin{minipage}[t]{.21\linewidth}
		\centering
		(0, 0, 0, 0),\\
		(0, 0, 0, 1),\\
		(0, 0, 1, 0),\\
		(0, 1, 0, 0),\\
		(1, 0, 0, 1),\\
		(1, 1, 1, 0).
		\subcaption{$P_{4,6,9}$}\label{fig:P469}
	\end{minipage}%
	\begin{minipage}[t]{.24\linewidth}
		\centering
		(0, 0, 0, 0, 0),\\
		(0, 0, 0, 0, 1),\\
		(0, 0, 0, 1, 0),\\
		(0, 0, 1, 0, 0),\\
		(0, 1, 0, 0, 0),\\
		(1, 0, 0, 0, 1),\\
		(1, 1, 1, 1, 0).
		\subcaption{$P_{5,7,12}$}\label{fig:P5712}
	\end{minipage}
	\begin{minipage}[t]{.29\linewidth}
		\centering
		(0, 0, 0, 0, 0, 0),\\
		(0, 0, 0, 0, 0, 1),\\
		(0, 0, 0, 0, 1, 0),\\
		(0, 0, 0, 1, 0, 0),\\
		(0, 0, 1, 0, 0, 0),\\
		(0, 1, 0, 0, 0, 0),\\
		(1, 0, 0, 0, 0, 1),\\
		(1, 1, 1, 1, 1, 0).
		\subcaption{$P_{6,8,15}$}\label{fig:P6815}
	\end{minipage}
	\caption{Examples of $d$-dimensional 2-neighborly 0/1-polytopes $P_{d,v,f}$ with $v = d+2$ vertices and $f$ facets.}\label{fig:01}
\end{figure}

The~combinatorial type of~a~$d$-polytope with $d+3$ vertices is defined by the~appropriate \emph{reduced Gale diagram} or wheel-sequence~(see~\cite[Sec.~6.3]{Grunbaum:2003} and~\cite{Fusy:2006}).
It is a set of points, placed at the center $C$ and vertices of a regular $2n$-gon, $n \ge 2$.
We enumearte the vertices from $0$ to $2n-1$ going clockwise (it does not matter what vertex is first).
The vertices and the center $C$ have nonnegative integer labels (multiplicities) $m_0$, $m_1$, \dots, $m_{2n-1}$, and $m(C)$.
In~the following we assume
\[
m_i \stackrel{\text{def}}{=} m_{i\bmod 2n}, \quad i \in \Z.
\]
In~particular, $m_{i+2n} \stackrel{\text{def}}{=} m_{i-2n} \stackrel{\text{def}}{=} m_i$.
A pair of opposite labels $m_i$ and $m_{i+n}$, $i\in[n]$, is called a \emph{diameter}.
The labels of a reduced Gale diagram have the following properties~\cite{Fusy:2006}:
\begin{enumerate}
	\item\label{P1}  $m(C) + m_0 + \dots + m_{2n-1} = d+3$.
	(The sum of all labels equals the number of vertices of~a~$d$-polytope.)
	\item $m_i + m_{i+n} > 0$ for every $i\in[n]$. 
	(A diagram has no diameter $\{0,0\}$.)
	\item $m_{i} + m_{i+1} > 0$ for every $i\in[2n]$. 
	(Two neighbours cannot both have label 0.)					
	\item\label{P4} $\sum_{j=i+1}^{i+n-1} m_{i} \ge 2$ for every $i\in[2n]$. 
	(The sum of labels in any open half-plane is not less than $2$.) 
\end{enumerate}
Two reduced Gale diagrams are combinatorially equivalent if one can be obtained from the other by a rotation or by a reflection.

A subset $S$ of these $d+3$ points (take into account the multiplicities) is called a \emph{cofacet}
if it is one of three types: 
a) the center $C$ itself, 
b) two opposite vertices $i$ and $i+n$ of the $2n$-gon, 
c) three vertices $\{i, j, k\}$ that form a triangle with $C$ in its interior.

\begin{theorem}{\cite[Sec.~6.3]{Grunbaum:2003}}
	\label{thm:GaleProperties}
	There are one-to-one correspondence between combinatorial types 
	of $d$-polytopes with $d+3$ vertices and reduced Gale diagrams with the properties \ref{P1}--\ref{P4} listed above.
	Moreover,
	\begin{description}
		\item[F] The number of facets of a polytope is equal to the number of cofacets of the appropriate reduced Gale diagram.
		\item[P] If $m(C) > 0$, then the polytope is a pyramid and its base is the same diagram with $m(C) \coloneqq m(C) - 1$. 
		\item[N] An appropriate polytope is 2-neighborly \textbf{iff} 
		$\sum_{j=i+1}^{i+n-1} m_{i} \ge 3$ for every $i\in[2n]$.
		(The sum of the labels in any open half-plane is at least 3.)
	\end{description}
\end{theorem}

\begin{lemma}\label{lem:1}
Decreasing a positive label $m_i$ of a 2-neighborly reduced Gale diagram by 1 leads to decreasing the number of cofacets at least by 3.
\end{lemma}
\begin{figure}
	\centering
\newcommand{\smallradius}{.07}
\newcommand{\point}{circle (\smallradius) }
\newcommand{\bradius}{2}
\newcommand{\Bradius}{2}
\newcommand{\diameter}[1]{\draw [dashed, thin] (#1:\bradius) -- ({180+#1}:\bradius);}
\tikzset{filled/.style={draw=black, fill=black}}
\tikzset{semifilled/.style={draw=black, fill=white}} 
\tikzset{empty/.style={draw=black, densely dotted, fill=white}}
		\begin{tikzpicture}[scale=0.7, fill=black, draw=black] 
\filldraw [empty, dotted] (0, 0) circle (\bradius);
\diameter{90}
\diameter{30}
\diameter{-60}
\filldraw [filled] (-90:\bradius) \point node[anchor=north] {$m_{i}$};
\filldraw [filled] (180+30:\bradius) \point node[anchor=north east] {$m_{j_{\min}}$};
\filldraw [filled] (120:\bradius) \point node[anchor=south east] {$m_{j_{\max}}$};
\filldraw [semifilled] (90:\bradius) \point node[anchor=south] {$m_{i+n}$};
\filldraw [semifilled] (30:\bradius) \point node[anchor=south west] {$m_{j_{\min}+n}$};
\filldraw [semifilled] (180+120:\bradius) \point node[anchor=north west] {$m_{j_{\max}}+n$};
\filldraw [draw=white, draw opacity = 0, fill=gray, fill opacity=0.3] %
(118:{0.1 + \bradius}) arc (120:-60:{0.1 + \bradius})  -- cycle;
\end{tikzpicture}
	\caption{The sum of lables in the gray semicircle is not less than 3.}\label{fig:2}
\end{figure}
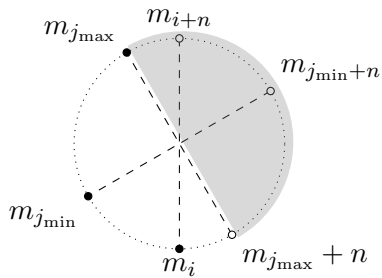
\begin{proof}
Let $m_i > 0$ and 
\begin{align*}
	j_{\min} &= \min \Set*{s \in \{i+1, \dots, i + n - 1\} \given m_s > 0},\\
	j_{\max} &= \max \Set*{s \in \{i+1, \dots, i + n - 1\} \given m_s > 0},
\end{align*}
where $n$ is the number of diameters of a diagram.
We have to show that there are at least 3 cofacets with a point at the vertex $i$ of the $2n$-gon.
Let us consider three types of cofacets (see Fig.~\ref{fig:2}):
\begin{enumerate}
	\item $\{i, i+n\}$ (there are exactly $m_{i+n}$ such cofacets for the given point);
	\item $\{i, j_{\min}, k\}$ for $i + n < k < j_{\min} + n$ (there are at least $\sum_{s = i+n+1}^{j_{\min} + n-1} m_s$ such cofacets);
	\item $\{i, j_{\max}, k\}$ for $j_{\min} + n \le k < j_{\max} + n$ (there are at least $\sum_{s = j_{\min}+n}^{j_{\max} + n-1} m_s$ of them).
\end{enumerate}
Obviously, these three sets of cofacets are pairwise distinct and the number of all of them is not less than
$\sum_{s = i+n}^{j_{\max} + n-1} m_s = \sum_{s = j_{\max} + 1}^{j_{\max} + n-1} m_s$.
But the last sum cann't be less than 3 by~the~property~N.
\end{proof}
	
Thus, if we are interested in considering only diagrams with a small (w.r.t. dimension) number of cofacets, then it is sufficient to find the minimal (by inclusion) diagrams. Evidently, minimal 2-neighborly diagrams have the following properties:
\begin{align*}
m(C) &= 0,\\
m_i &\le 3, \quad i\in[2n],\\ 
n   &\le 7. 
\end{align*}
It is easy to enumerate all such diagrams by a computer.
There are 4 such diagrams (we give vectors $(m_0, \dots, m_{2n-1})$):
\begin{enumerate}
	\item $(3, 3, 3, 3)$ encodes a 9-polytope $P_{9, 12, 18}$ with 12 vertices and 18 facets.
	\item $(1, 2, 1, 2, 1, 2)$ encodes a 6-polytope $P_{6, 9, 15}$ with 9 vertices and 15 facets.
	\item $(1, 1, 1, 1, 1, 1, 1, 1)$ encodes a 5-polytope $P_{5, 8, 12}$ with 8 vertices and 12 facets.
	\item $(0, 1, 1, 1, 0, 1, 1, 1, 1, 1)$ encodes a 5-polytope $P_{5, 8, 14}$ with 8 vertices and 14 facets.
\end{enumerate}
For the first, second and forth ones, we have ``facets $=$ $d+9$''.
From Lemma~\ref{lem:1} we know, that adding a point to such a diagram increases the number of (co)facets at least by 3. Hence, some new nonminimal diagrams with ``facets $\le$ $d+9$'' could be obtained only by adding points to the third diagram. It is easy to verify, that adding a point to the third diagram increases the number of (co)facets at least by 4. Therefore, we listed all combinatorial types of 2-neighborly $d$-polytopes with $d+3$ vertices and at most $d+9$ facets. Surely, we have to add $k$-fold pyramids, $k \in \N$, over them to this list.

It is interesting to remark, that all these polytopes can be realized as 0/1-polytopes (see Fig.~\ref{fig:3}). 
\begin{figure}
	\begin{minipage}[t]{.19\linewidth}
		\centering
		(0, 0, 0, 0, 0),\\
		(0, 0, 0, 0, 1),\\
		(0, 0, 0, 1, 0),\\
		(0, 0, 1, 0, 0),\\
		(0, 1, 0, 0, 1),\\
		(0, 1, 1, 1, 0),\\
		(1, 0, 0, 1, 0),\\
		(1, 0, 1, 0, 1).
		\subcaption{$P_{5,8,12}$}\label{fig:P5812}
	\end{minipage}%
	\begin{minipage}[t]{.19\linewidth}
		\centering
		(0, 0, 0, 0, 0),\\
		(0, 0, 0, 0, 1),\\
		(0, 0, 0, 1, 0),\\
		(0, 0, 1, 0, 1),\\
		(0, 1, 0, 1, 1),\\
		(0, 1, 1, 0, 0),\\
		(1, 0, 0, 1, 1),\\
		(1, 1, 0, 0, 0).
		\subcaption{$P_{5,8,14}$}\label{fig:P5814}
	\end{minipage}
	\begin{minipage}[t]{.22\linewidth}
		\centering
		(0, 0, 0, 0, 0, 0),\\
		(0, 0, 0, 0, 0, 1),\\
		(0, 0, 0, 0, 1, 0),\\
		(0, 0, 0, 1, 0, 0),\\
		(0, 0, 1, 0, 0, 0),\\
		(0, 1, 0, 0, 0, 0),\\
		(1, 0, 0, 0, 0, 1),\\
		(1, 0, 0, 1, 1, 0),\\
		(1, 1, 1, 0, 0, 0).
		\subcaption{$P_{6,9,15}$}\label{fig:P6915}
	\end{minipage}
\begin{minipage}[t]{.31\linewidth}
	\centering
	(0, 0, 0, 0, 0, 0, 0, 0, 0),\\
	(0, 0, 0, 1, 0, 0, 0, 0, 0),\\
	(0, 0, 1, 0, 0, 0, 0, 0, 0),\\
	(0, 1, 0, 0, 0, 0, 0, 0, 0),\\
	(1, 0, 0, 1, 0, 0, 0, 0, 0),\\
	(1, 1, 1, 0, 0, 0, 0, 0, 0),\\
	(0, 0, 0, 0, 1, 0, 0, 0, 0),\\
	(0, 0, 0, 0, 1, 0, 0, 0, 1),\\
	(0, 0, 0, 0, 1, 0, 0, 1, 0),\\
	(0, 0, 0, 0, 1, 0, 1, 0, 0),\\
	(0, 0, 0, 0, 1, 1, 0, 0, 1),\\
	(0, 0, 0, 0, 1, 1, 1, 1, 0).
	\subcaption{$P_{9,12,18}$}\label{fig:P91218}
\end{minipage}
	\caption{Examples of $d$-dimensional 2-neighborly 0/1-polytopes $P_{d,v,f}$ with $v = d+3$ vertices and $f$ facets.}\label{fig:3}
\end{figure}

\section{Incidence matrices of 2-neighborly polytopes}  
\label{sec:IncMat}

Let $P$ be a $d$-polytope with $n$ vertices $\{v_1, \dots, v_n\}$ and $k$ facets $\{F_1, \dots, F_k\}$.
\emph{A facet-vertex incidence matrix} of $P$ is a matrix $M \in \{0,1\}^{k \times n}$ with $M_{ij} = 1$ for $v_j \in F_i$ and $M_{ij} = 0$ otherwise.
In the following, we say that 0/1-row (column) $x$ is \emph{a subset} of a 0/1-row (column) $y$ if
the set of 1's in $x$ is a subset of the set of 1's in $y$.

\SetKwProg{Proc}{Procedure}{}{end} 
\SetKwProg{Fn}{Function}{}{end} 
\SetKwInOut{Input}{Input}
\SetKwInOut{Output}{Output}
\SetKwFunction{GetF}{get\_facet}%
\SetKwData{Mat}{matrix}
\begin{algorithm}
	\caption{The extraction of the incidence matrix of a facet from the given 0/1-matrix} 
	\label{alg:facet}
	\SetKwData{i}{i}
	\SetKwData{A}{A}
	\Input{\Mat and a row number \i}
	\Output{the matrix \A extracted from the \i-th row of \Mat}
	\BlankLine
	\Fn{\GetF{\Mat, \i}}{
		\A is an empty matrix\;
		\For{every column $c$ in \Mat}{
			\lIf{$c[\i] = 1$}{append $c$ to \A}
		}
		\For{every row $r$ in \A}{
			\lIf{$r$ is a subset of another row in \A}{remove $r$ from \A}
		}
		\Return{\A}\;
	}
\end{algorithm}

A facet-vertex incidence matrix $M$ of $P$ contains all information about the combinatorial structure of $P$.
In particular, a facet-vertex incidence matrix of a $i$-th facet of $P$ can be extracted from $M$ by a~2-stage process (see Algorithm~\ref{alg:facet}). 
Firstly, we have to remove from $M$ the columns with 0's in the $i$-th row. 
Finally, we remove from $M$ every row that is a subset of some other row.
I.e., an incidence matrix of a facet of $P$ is a submatrix (up to a permutation of rows and columns) of an incidence matrix of $P$ (see, e.g., Fig.~\ref{fig:facet}).
\begin{figure}
	\centering
	\begin{tabular}{c|lcc|}
		\multicolumn{1}{c}{} & 
		$v_1$ $v_2$ \hfill \dots \hfill{} & $v_{n-1}$ & 
		\multicolumn{1}{c}{$v_n$}\\
		\cline{2-4}
		\begin{tabular}{c}
			$F_1$\\
			$F_2$\\
			\vdots\\
		\end{tabular}
		&
		\multicolumn{1}{c|}{
			\parbox[c]{3cm}{\centering An incidence\\ matrix of $F_k$}
		}
		&	
		\begin{tabular}{c}
			$\star$\\ $\star$\\	\vdots\\
		\end{tabular}
		&
		\begin{tabular}{c}
			$\star$\\	$\star$\\	\vdots\\
		\end{tabular}
		\\
		\cline{2-2}
		& \,$\star$\; $\star$ \hfill \dots \hfill$\star$\;  $\star$ & $\star$ & $\star$ \\ 
		$F_k$ &  \,1\; 1 \hfill \dots \hfill 1\; 1 & 0 & 0\\
		\cline{2-4}
	\end{tabular}
	\caption{A facet-vertex incidence matrix $M$ of a polytope $P$ and an incidence matrix of its facet $F_k$}\label{fig:facet}
\end{figure}

2-Neighborlyness of a polytope $P$ means that the intersection of any two columns of a matrix $M$ is not a subset of any other column. In the following, if a matrix has this property, we call it \emph{2-neighborly}.
Let us list some obvious properties of an incidence matrix $M$ of a $d$-polytope:
\begin{enumerate}
	\item Every row and every column of $M$ has at least $d$ 1's.
	\item Any row is not a subset of another row. The same is true for columns.
\end{enumerate}
Thus, if we are looking for a polytope with a given number of vertices, facets, and some other combinatorial properties, then we can try to enumerate all 0/1-matrices with these properties.

It is known, that a 2-neighborly 4-polytope with $v$ vertices has exactly $v(v-3)/2$ facets.
Hence, there are only two combinatorial types of such polytopes satisfying the restriction ``facets${} \le d+9$'': the simplex $P_{4,5,5}$ and the polytope $P_{4,6,9}$.
Suppose, we want to list all combinatorial types of 2-neighborly 5-polytopes with facets${} \le d+9$.
Then, all facets of such a polytope $P$ must be combinatorially equivalent to $P_{4,5,5}$ or $P_{4,6,9}$ (otherwise, the number of facets of $P$ will be greater than 14).
For example, we want to find such a polytope $P$ with 7 vertices and 10 facets.
So, we can start from the matrix on Fig.~\ref{fig:facet}, where $n = 7$, $k = 10$, and an incident matrix of $F_k$ is a matrix for $P_{4,5,5}$ or $P_{4,6,9}$. 
Replacing the stars with zeros and ones, and checking the listed above properties of the matrix, we will finally find all suitable combinatorial types. Of course, there may be present nonrealizable types, but we have not found such examples in our experiments.


\SetKwData{Fl}{Flist} 
For enumerating all 0/1-matrices with these properties by a computer, we use Algorithms~\ref{alg:facet}--\ref{alg:enumeration2}.
\SetKwData{Pl}{Plist}
\SetKwData{fc}{feas\_cols}
\SetKwData{num}{m}
\SetKwArray{isf}{isF}
\SetKwFunction{addV}{add\_vertex}%
\begin{algorithm}
	\caption{The~enumeration of~facet-vertex incidence matrices with a given facet} 
	\label{alg:enumeration}
	\SetKwData{F}{F} 
	\SetKwData{Dim}{dim}
	\SetKwData{Vrt}{vrt}
	\SetKwData{Fct}{fct}
	\SetKwData{mfv}{minfv}
	\SetKwData{fr}{feas\_rows}
	\SetKwFunction{Find}{find\_matrices}%
	\SetKwFunction{GenFR}{gen\_feasible\_rows}%
	\SetKwFunction{GenFC}{gen\_feasible\_columns}%
	\Input{ a list of incidence matrices of facets \Fl, a matrix \F from \Fl, a~dimension \Dim, a~number of vertices (columns) \Vrt, a~number of facets (rows) \Fct, and a~minimal number of facets incident to a~vertex \mfv}
	\Output{ the list \Pl of incidence matrices of~2-neighborly \Dim-polytopes with \Vrt vertices and \Fct facets (one of the facets is \F)}
	\BlankLine
	\Fn{\Find{\Fl, \F, \Dim, \Vrt, \Fct, \mfv}}{
		$M \coloneqq \F$\;
		$\fr \coloneqq$ \GenFR{\F, $\Dim - (\Vrt - \vertices(\F))$}\;
		$m \coloneqq \Fct - \facets(\F) - 1$\;
		\For{every multisubset $S$ of size $m$ from \fr}{
			$M' \coloneqq M$\;
			append the rows from $S$ to $M'$\;
			append the row $(1,\dots,1)$ to $M'$\;
			\If{the number of 1's in a column in $M'$ less than \mfv}{
				\tcp{the subset $S$ is not suitable}
				continue\;
			}	
			$\fc \coloneqq$ \GenFC{$M'$, \mfv}\;
			$\isf = (0,\dots,0,1)^T$\;
			\addV{$M'$, $\Vrt - \vertices(\F)$, \fc, \Fl, \Pl, \isf}\;
		}
		canonize matrices in \Pl\;
		sort \Pl and remove duplicates\;
		\Return{\Pl}\;
	}
	\BlankLine
	\Fn{\GenFR{\Mat, \num}}{
		Return the set of 0/1-rows of the length $\cols(\Mat)$ with the properties\;
		a) the number of 1's is not less than \num\;
		b) a row must be a proper subset of some row in \Mat\;
	}
	\BlankLine
	\Fn{\GenFC{\Mat, \num}}{
		Return the set of 0/1-columns of length $\rows(\Mat)$ with the properties\;
		a) the last entry is 0\;
		b) the number of 1's is not less than \num\;
		c) \Mat with the column is 2-neighborly\;
	}
\end{algorithm}
\begin{algorithm}
	\caption{Continuation of the algorithm \ref{alg:enumeration}} 
	\label{alg:enumeration2}
	\SetKwArray{isfn}{isFnew}
	\Input{a~\Mat, a~number of additional columns \num, a~list of feasible columns \fc, a~list of facets \Fl, a~list of good matrices \Pl, and a~0/1-column \isf}
	\BlankLine
	\Proc{\addV{\Mat, \num, \fc, \Fl, \Pl, \isf}}{
		\eIf{$\num > 1$}{
			\For{$c \in \fc$}{
				\lIf{$c \cap \isf \neq \emptyset$}{continue}
				\lIf{$\Mat \cup c$ is not 2-neighborly}{continue}
				$\isfn \coloneqq \isf$\;
				\For{$i \in [\rows(\Mat)]$}{
					\If{$c[i] = 1$}{
						$A \coloneqq$ \GetF{$\Mat \cup c, i$}\;
						\lIf{$A \in \Fl$}{$\isfn{i} \coloneqq 1$}
					}
				}
				\addV{$\Mat \cup c$, $\num - 1$, $\fc \setminus c$, \Fl, \Pl, \isfn}\;
			}
		}
		{
			$c \coloneqq \bm{1} - \isf$\;
			$\Mat \coloneqq \Mat \cup c$\;
			\lIf{$\Mat$ is not 2-neighborly}{\Return}
			\lIf{there is a row in $\Mat$ that is a subset of some other row}{\Return}
			\For{$i \in [\rows(\Mat)]$}{
				\If{$c[i] = 1$}{
					$A \coloneqq$ \GetF{$\Mat, i$}\;
					\lIf{$A \in \Fl$}{\Return}
				}
			}
			append \Mat to \Pl\;
		}	
	}
\end{algorithm}
The input of the algorithm is: a~list \Fl of incidence matrices of feasible facets (2-neighborly $(d-1)$-polytopes), a~matrix from the list (a~must-have facet), a~dimension~$d$, a~number of vertices, a~number of facets, and a~minimal number of 1's in a~column of a~resulting matrix (for a 2-neighborly polytope $P$, this number is equal to $d$ only if $P$ is a simplex).
The output is a list of incidence matrices of 2-neighborly $d$-polytopes with these properties.
In the realization of the algorithm, we use some additional optimizations.
For example, the checking if a given matrix $A$ is in the list \Fl is computationally expensive, since we have to compare matrices modulo a permutation of rows and columns.
It may be done by computing the~canonical form of~a row-column digraph of~a~matrix by using \emph{bliss}~\cite{bliss}. 
But we have found the faster method. 
For every $M \in \Fl$, we make a preprocessing: for every permutation of columns in $M$ we lexicographically sort rows and keep all these matrices in the set \Fl (without duplicates).
When we need to check $A \in \Fl$, it is sufficient to sort rows of $A$ and check $A \in \Fl$ by a binary search.
Our experiments show that this works ten times faster than~\emph{bliss}.
The computer program is available at \url{https://github.com/maksimenko-a-n/2neighborly-inc-matrices}.

The results of using this program are described in the next section.

\section{d-Polytopes with at least d + 4 vertices}  
\label{sec:d+4}

\subsection{Vertex figures of a 2-neighborly polytope}  

\emph{A vertex figure} of a polytope $P$ is an intersection of $P$ with a hyperplane $H$, where $H$ strictly separates one vertex of $P$ from all others.

\begin{lemma}
	\label{lem:vfig}
	A vertex figure of a 2-neighborly $d$-polytope $P$ with $n$ vertices is a 2-simplicial $(d-1)$-polytope $Q$ with $n-1$ vertices.
\end{lemma}

The correctness of the lemma follows from the 3-simpliciality of a 2-neighborly polytope.

\begin{lemma}
	For $d \ge 3$, a 2-simplicial $d$-polytope with at least $d+3$ vertices has at least $d+4$ facets.
\end{lemma}
\begin{proof}
	The lemma is a consequence of the following statement.
	If $P$ is a 2-simplicial $d$-polytope with at most $d+3$ facets, then it is a $d$-simplex or a $(d-3)$-fold pyramid over a triangular bipyramid.
	We prove it by induction on $d$.
	
	It is easy to check, that there are only two 2-simplicial 3-polytopes with at most 6 facets: a~tetrahedron and a triangular bipyramid $B$.
	Hence, for $d = 3$ the statement is true. It is also easy to check for $d = 4$.
	Suppose now, that the statement is true for $d = k \ge 4$.
	Consider a~2-simplicial $(k+1)$-polytope $P$ with at most $k+4$ facets.
	Every its facet have to be a 2-simplicial $k$-polytope with at most $k+3$ facets.
	By the induction hypothesis, there are only two such polytopes: a $k$-simplex and a $(k-3)$-fold pyramid over $B$.
	If $P$ has no a facet that is a $(k-3)$-fold pyramid over $B$, then $P$ is a $(k+1)$-simplex or $P$ is simplicial and has at least $2k+2 > k + 4$ facets, by the lower bound theorem~\cite{Barnette:1971}.
	Else, if a $(k-3)$-fold pyramid over $B$ is a facet $F$ of $P$, then $P$ is a pyramid over $F$ or every vertex of $P$ is not incident with at least two facets.
	In the last case, we can consider an apex of $F$ (since $F$ is a pyramid). 
	It is incident with $k+3$ facets of $P$ and is not incident with at least two. 
	Hence, the number of facets of $P$ is at least $k+5$.
\end{proof}

\begin{corollary}
	\label{cor:1}
	Every vertex of a 2-neighborly $d$-polytope with at least $d+3$ vertices is incident with at least $d+3$ facets.
\end{corollary}

\subsection{4- and 5-Polytopes with at most 14 facets}  

There are only three combinatorial types of 2-neighborly 4-polytopes with at most 14 facets:
$4$-simplex, a polytope $P_{4,6,9}$, and a cyclic polytope $P_{4,7,14}$ with 7 vertices and 14 facets.

All combinatorial types of 5-polytopes with at most 9 vertices enumerated by Fukuda, Miyata, Moriyama~\cite{Fukuda:2013}. There are no 2-neighborly polytopes with at most 14 facets, except the ones listed in Section~\ref{sec:Gale}: a simplex, 
a pyramid over $P_{4,6,9}$, a simplicial polytope $P_{5,7,12}$,
$P_{5,8,12}$, and $P_{5,8,14}$.

If a 2-neighborly 5-polytope $P$ has at most 14 facets, then it has only two combinatorial types of facets: a $4$-simplex and a polytope $P_{4,6,9}$. (Every facet of $P$ must be 2-neighborly 4-polytope with at most 13 facets.)
Suppose, that $P$ has at least 10 vertices.
By Lemma~\ref{lem:vfig}, every its vertex figure is a 2-simplicial 4-polytope $Q$ with at least 9 vertices.
It is easy to prove, that $Q$ has at least 9 facets.
Hence, every vertex of $P$ is incident with at least 9 facets.
Then $P$ has at least $10 \cdot 9 / 6 = 15$ facets.

\subsection{6-Polytopes with at most 15 facets}  

Now, let $P$ be a 2-neighborly 6-polytope with at least 10 vertices and \emph{at most 15 facets}.
Hence, every facet of $P$ is combinatorially equivalent to one of five polytopes:
a simplex $P_{5,6,6}$, a pyramid over $P_{4,6,9}$, a simplicial polytope $P_{5,7,12}$, $P_{5,8,12}$, and $P_{5,8,14}$ (see Figures~\ref{fig:01} and~\ref{fig:3}).
First of all, we investigate properties of vertex figures of $P$.
We have analyzed the database of 5-polytopes with at most 9 vertices~\cite{Fukuda:2013} and have found the following:
\begin{enumerate}
	\item All 2-simplicial 5-polytopes with 9 vertices have at least 9 facets. Only one polytope has exactly 9 facets and it is 2-simple. Its facet-vertex incidence matrix is shown in Fig.~\ref{fig:99}. All such polytopes with 9 or 10 facets have at least one facet with 7 vertices.
	\item There are no 2-simplicial 5-polytope with at most 9 facets and at least 10 vertices. (To establish this, we searched for 2-simple polytopes with at most 9 vertices.)
\end{enumerate}
\begin{figure}
	\centering
	0 0 1 1 1 1 1 1 1\\
	1 1 0 0 1 1 1 1 1\\
	1 1 1 1 0 0 1 1 1\\
	0 1 1 1 1 1 0 0 1\\
	1 1 0 1 1 1 0 1 0\\
	1 1 1 1 0 1 1 0 0\\
	0 1 0 1 1 0 0 1 1\\
	0 1 1 0 0 1 1 0 1\\
	1 0 0 1 0 1 1 1 0
	\caption{The facet-vertex incidence matrix of a 2-simplicial 5-polytope with 9~vertices and 9~facets}\label{fig:99}
\end{figure}

The first statement means that if a 2-neighborly 6-polytope $P$ has exactly 10 vertices and every its facet has at most 7 vertices, then every its vertex figure has at least 11 facets.
Thus, 
the number of facets of $P$ cann't be less than $\lceil 10 \cdot 11 / 7 \rceil = 16$.
From the second statement, one can deduce the same lower bound for a polytope $P$ with more than 10 vertices. Hence, a 2-neighborly 6-polytope $P$ with at least 10 vertices and at most 15 facets must have at least one facet with 8 vertices. There are only two types of such facets: $P_{5,8,12}$ and $P_{5,8,14}$.

The fact that ``there are no 2-simplicial 5-polytope with at most 9 facets and at least 10 vertices'' means that if a 2-neighborly 6-polytope $P$ has at least 11 vertices, then every its vertex is incident to at least 10 facets. 
Thus, for polytopes with at least 13 vertices, the number of facets cann't be less than $\lceil 13\cdot 10/8 \rceil = 17$.
The cases, when $P$ has 10--12 vertices and at most 15 facets (and at least one of facets is $P_{5,8,12}$ or $P_{5,8,14}$), we check with the computer program described in the previous section. There are only two combinatorial types with these properties: $P_{6,10,14}$ can be realized as a 0/1-polytope and $P_{6,10,15}$ cann't be realized as a 0/1-polytope  (see Fig.~\ref{fig:d6-7}).

\begin{figure}
	\begin{minipage}[t]{.19\linewidth}
		\centering
(0, 0, 0, 0, 0, 0),\\
(0, 0, 0, 0, 0, 1),\\
(0, 0, 0, 0, 1, 0),\\
(0, 0, 0, 1, 0, 0),\\
(0, 0, 1, 0, 0, 0),\\
(0, 1, 0, 0, 0, 0),\\
(1, 0, 0, 0, 1, 1),\\
(1, 0, 1, 1, 0, 0),\\
(1, 1, 0, 1, 0, 1),\\
(1, 1, 1, 0, 1, 0).
		\subcaption{$P_{6,10,14}$}\label{fig:P61014}
	\end{minipage}
\qquad\qquad
	\begin{minipage}[t]{.19\linewidth}
(1, 1, 1, 1, 1, 0),\\
(1, 1, 1, 1, -1, 0),\\
(1, 1, 1, -1, 1, 0),\\
(1, 1, -1, 1, 1, 0),\\
(1, -1, 1, 1, -1, 0),\\
(1, -1, -1, -1, 1, 0),\\
(-1, 1, 1, -1, 1, 0),\\
(-1, 1, -1, 1, -1, 0),\\
(0, 0, 1, 1, 0, 1),\\
(0, 0, -1, 0, 1, 1).
	\subcaption{$P_{6,10,15}$}\label{fig:P61015}
	\end{minipage}
\qquad\qquad
\begin{minipage}[t]{.22\linewidth}
	\centering
	(0, 0, 0, 0, 0, 0, 0),\\
	(0, 0, 0, 0, 0, 0, 1),\\
	(0, 0, 0, 0, 0, 1, 0),\\
	(0, 0, 0, 0, 1, 0, 0),\\
	(0, 0, 0, 1, 0, 0, 0),\\
	(0, 0, 1, 0, 0, 0, 0),\\
	(0, 1, 0, 0, 0, 0, 0),\\
	(1, 0, 0, 0, 0, 1, 1),\\
	(1, 0, 0, 1, 1, 0, 0),\\
	(1, 0, 1, 0, 1, 0, 1),\\
	(1, 0, 1, 1, 0, 1, 0),\\
	(1, 1, 0, 0, 1, 1, 0),\\
	(1, 1, 0, 1, 0, 0, 1),\\
	(1, 1, 1, 0, 0, 0, 0).
	\subcaption{$P_{7,14,16}$}\label{fig:P71416}
\end{minipage}
	\caption{2-neighborly $d$-polytopes with at least $d+4$ vertices.}\label{fig:d6-7}
\end{figure}

Thus, there are exactly 9 combinatorial types of 2-neighborly 6-polytopes with at most 15 facets:
a simplex $P_{6,7,7}$,
a 2-fold pyramid over $P_{4,6,9}$,
a pyramid over $P_{5,7,12}$,
a simplicial polytope $P_{6,8,15}$,
a pyramid over $P_{5,8,12}$,
a pyramid over $P_{5,8,14}$,
$P_{6,9,15}$, 
$P_{6,10,14}$, 
$P_{6,10,15}$. 

\subsection{7-Polytopes with at most 16 facets}  

In section~\ref{sec:Gale}, there are listed all combinatorial types of 2-neighborly $d$-polytopes with at most $d+3$ vertices and $d+9$ facets. In this section we suppose that a 2-neighborly 7-polytope $P$ has at least 11 vertices and at most 16 facets. At first, we show that $P$ has a facet with at least 9 vertices.

\begin{lemma}
	\label{lem:d+1}
	If $P$ is a 2-neighborly $d$-polytope wiht at least $d+4$ vertices and every its facet has at most $d+1$ vertices, then $P$ has at least $2d + 3$ facets.
\end{lemma}

\begin{proof}
	If $P$ is simplicial, then $\facets (P) \ge (d-1)(\vertices(P) - d) + 2 \ge 4d - 2$, by the lower bound theorem~\cite{Barnette:1971}.
	Suppose, $P$ has a facet $Q$ with $d+1$ vertices. Hence, $\facets(Q) \ge d + 4$.
	Moreover, there are three vertices $x, y, z \in P \setminus Q$, that are separated from the other vertices of $P$. Thus, \[\conv\{x,y,z\} \cap \conv(\vertices P \setminus \{x,y,z\}) = \emptyset.\]
	Let $A$ be an affine hull of $\{x,y,z\}$. Thus, $\dim A = 2$ and $A \cap P = \conv\{x,y,z\}$, since every pair of vertices of $P$ form a 1-face of $P$. Therefore, $T = \conv\{x,y,z\}$ is a~2-face of $P$.
	Hence, there are at least $d-2$ facets of $P$ that contains $T$. Every such a facet $F$ has $\le d+1$ vertices, and three of them are $x, y, z$. Consequently, $F$ does not have a common ridge ($(d-2)$-face) with $Q$. Therefore,
	\[
	\facets(P) \ge \facets(Q) + 1 + d-2 \ge 2d+3.
	\]
\end{proof}

In the following, we consider only $d$-polytopes that have at least one facet with $d+2$ vertices.

\begin{lemma}
	If $P$ is a 2-neighborly $d$-polytope and $P$ is not a pyramid, then for every vertex $v$ there are at least 3 facets that is not incident to $v$.
\end{lemma}

\begin{proof}
Let $P$ has $n$ vertices and $k$ facets.
$P$ is not a pyramid. Hence, every facet of $P$ is incident to at most $n-2$ vertices and every vertex is incident to at most $k-2$ facets. Suppose, a vertex $v$ is not incident to exactly two facets $F_1$ and $F_2$.
Since $F_1$ and $F_2$ has at most $n-2$ vertices, then there are vertices $w_1$ and $w_2$ such that $w_1 \notin F_1$, $w_2 \notin F_2$, $w_1 \ne v$, $w_2 \ne v$.
If $w_1 = w_2$, then for every facet $F$ of $P$ we have $w_1 \in F$ $\Rightarrow$ $v \in F$, and $w_1$ is not a vertex. Thus, $w_1 \ne w_2$, and for every facet $F$ of $P$ we have $w_1, w_2 \in F$ $\Rightarrow$ $v \in F$. Therefore, vertices $w_1$ and $w_2$ do not form an edge of $P$, and $P$ is not 2-neighborly.
\end{proof}

\begin{corollary}
	\label{cor:pyr}
	Let $Q$ be a pyramid and a 2-neighborly $(d-1)$-polytope with $k$ facets.
	If $P$ is a 2-neighborly $d$-polytope and $Q$ is a facet of $P$, then $P$ is a pyramid or $P$ has at least $k+3$ facets.
\end{corollary}

From the corollary, one can conclude, that if a 2-neighborly 7-polytope $P$ with at most 16 facets has a pyramid over $P_{5,8,14}$ as a facet, then $P$ is a 2-fold pyramid over $P_{5,8,14}$.

Suppose now, that $P$ is not a pyramid and $P$ has no facets with 9 or more vertices, except a pyramid over $P_{5,8,12}$.
Then, its incidence matrix looks like in Fig.~\ref{fig:case5812} (remember, that an incidence matrix of a non-pyramid has at least 3 zeroes in every column). 
\begin{figure}
	\centering
	\begin{tabular}{l|*{12}{c}|}
\multicolumn{1}{c}{}
& $v_1$ & $v_2$ & $v_3$ & $v_4$ & $v_5$ & $v_6$ & $v_7$ & $v_8$ & $v_9$ & $v_{10}$ & \dots & \multicolumn{1}{c}{$v_{n}$}\\
\cline{2-13}
$F_1$ & 0 & 0 & 1 & 1 & 1 & 1 & 1 & 1 & 1 & $\star$ & $\star$ & $\star$\\
$F_2$ & 1 & 1 & 0 & 0 & 1 & 1 & 1 & 1 & 1 & $\star$ & $\star$ & $\star$\\ 
$F_3$ & 1 & 1 & 1 & 1 & 0 & 0 & 1 & 1 & 1 & $\star$ & $\star$ & $\star$\\
$F_4$ & 1 & 1 & 1 & 1 & 1 & 1 & 0 & 0 & 1 & $\star$ & $\star$ & $\star$\\
$F_5$ & 0 & 1 & 0 & 1 & 0 & 1 & 1 & 1 & 1 & $\star$ & $\star$ & $\star$\\
$F_6$ & 1 & 0 & 1 & 0 & 1 & 0 & 1 & 1 & 1 & $\star$ & $\star$ & $\star$\\
$F_7$ & 0 & 1 & 0 & 1 & 1 & 1 & 0 & 1 & 1 & $\star$ & $\star$ & $\star$\\
$F_8$ & 0 & 1 & 1 & 1 & 1 & 0 & 0 & 1 & 1 & $\star$ & $\star$ & $\star$\\
$F_9$ & 1 & 1 & 1 & 0 & 1 & 0 & 0 & 1 & 1 & $\star$ & $\star$ & $\star$\\
$F_{10}$ & 1 & 0 & 1 & 0 & 1 & 1 & 1 & 0 & 1 & $\star$ & $\star$ & $\star$\\
$F_{11}$ & 1 & 0 & 1 & 1 & 0 & 1 & 1 & 0 & 1 & $\star$ & $\star$ & $\star$\\
$F_{12}$ & 1 & 1 & 0 & 1 & 0 & 1 & 1 & 0 & 1 & $\star$ & $\star$ & $\star$\\
$F_{13}$ & 1 & 1 & 1 & 1 & 1 & 1 & 1 & 1 & 0 & $\star$ & $\star$ & $\star$\\
$F_{14}$ & 1 & 1 & 1 & 1 & 1 & 1 & 1 & 1 & 1 & 0 & 0 & 0\\
$F_{15}$ & $\star$ & $\star$ & $\star$ & $\star$ & $\star$ & $\star$ & $\star$ & $\star$ & 0 & $\star$ & $\star$ & $\star$\\
$F_{16}$ & $\star$ & $\star$ & $\star$ & $\star$ & $\star$ & $\star$ & $\star$ & $\star$ & 0 & $\star$ & $\star$ & $\star$\\
\cline{2-13}
	\end{tabular}
	\caption{An incidence matrix of a 7-polytope $P$, one of whose facets is a pyramid over $P_{5,8,12}$}\label{fig:case5812}
\end{figure}
The vertices $v_1$--$v_8$ and facets $F_1$--$F_{12}$ form an incidence matrix of $P_{5,8,12}$.
Every row of the matrix must have at most nine 1's.
A row $F_{13}$ is not a subset of $F_{14}$. Hence, exactly one of the stars in the row $F_{13}$ is equal to~1. Without loss of generality, $v_{10} \in F_{13}$, $v_{i} \notin F_{13}$, $i > 10$.
Thus, the intersection of $F_{13}$ and $F_{14}$ is a $P_{5,8,12}$. Consequently, $F_{13}$ is a pyramid with base $P_{5,8,12}$ and apex $v_{10}$. Then, $v_{10}$ is incident with all facets except $F_{14}$--$F_{16}$.
Hence, the intersection of columns $v_{10}$ and $v_{n}$ is a subset of $v_9$.
Therefore, $P$ is not 2-neighborly. Summarizing these arguments, one can come to the following conclusion.

\begin{proposition}
	\label{prop:58}
	Let $Q$ be a $k$-fold pyramid over $P_{5,8,12}$, $k \in \N$.
	Let $P$ be a 2-neighborly $(k+6)$-polytope with at most $k+15$ facets,
	and $P$ has no facets with $k+8$ or more vertices, except combinatorially equivalent to $Q$.
	Then $P$ is a pyramid over $Q$.
\end{proposition}

So, we need to check if there are 2-neighborly 7-polytopes $P$ with at most 16 facets and at least one facet of type $P_{6,9,15}$ (not a pyramid over $P_{5,8,14}$), $P_{6,10,14}$, or $P_{6,10,15}$.

Suppose, a 2-neighborly 7-polytope $P$ has 16 facets, one of facets is $P_{6,9,15}$ and $P$ has no facets with 10 or more vertices. An incidence matrix $M$ of $P$ has an incidence matrix of $P_{6,9,15}$ as a submatrix. Every row in $M$ has at most 9 ones. Totally, $M$ has at most $16 \cdot 9 = 144$ ones. The first 9 columns in $M$ has $96 + 9 = 105$ ones (from a matrix of $P_{6,9,15}$ and from the last row of $M$). Every other column has at least 10 ones. (By Corollary~\ref{cor:1}, every vertex of $P$ is incident with at least 10 facets.)
Therefore, $M$ has at most $\lfloor (144 - 105) / 10 \rfloor = 3$ additional columns.
By the computer program, we enumerate all such matrices and have found that only a pyramid over $P_{6,9,15}$ can be a 2-neighborly 7-polytope with such properties.

The same we do for 7-polytopes with facets $P_{6,10,14}$ and $P_{6,10,15}$.
Except pyramids, we have found only one 2-neighborly 7-polytope with at most 16 facets.
It has 14 vertices and can be realized as a 0/1-polytope (see Fig.~\ref{fig:P71416}). 

Therefore, there are 10 combinatorial types of 2-neighborly 7-polytopes with at most 16 facets: nine pyramids over 2-neighborly 6-polytopes and a polytope $P_{7,14,16}$.

\subsection{d-Polytopes with at most d+9 facets for d > 7}  

In this section, we consider only 2-neighborly $d$-polytopes with at least $d+4$ vertices and at most $d+9$ facets.
From Lemma~\ref{lem:d+1}, we know, that such polytopes have a facet with at least $d+2$ vertices.
From Corollary~\ref{cor:pyr}, if such a polytope $P$ has at least one facet, combinatorially equivalent to $k$-fold pyramids over $P_{5,8,14}$, $P_{6,9,15}$, $P_{6,10,14}$, or $P_{6,10,15}$, then $P$ is a pyramid.
From Proposition~\ref{prop:58}, if $P$ has no facets with at most $d+2$ vertices and all facets with exactly $d+2$ vertices combinatorially equivalent to a $k$-fold pyramid over $P_{5,8,12}$, then $P$ is a $(k+1)$-fold pyramid over $P_{5,8,12}$.

There are only two unconsidered cases left:
8-polytopes with $P_{7,14,16}$ as a facet and 10-polytopes with $P_{9,12,18}$ as a facet.

Suppose, a 2-neighborly 8-polytope $P$ has at most 17 facets, and one of facets is combinatorially equivalent to $P_{7,14,16}$. Hence, all other 16 facets of $P$ have a common 6-faces with $P_{7,14,16}$.
Two 6-faces of $P_{7,14,16}$ a simplices, and 14 are combinatorially equivalent to $P_{6,10,14}$.
We already now, that only two 2-neighborly 7-polytopes with at most 16 facets has $P_{6,10,14}$ as a facet:
$P_{7,14,16}$ and a pyramid over $P_{6,10,14}$. 
We already know, that if the pyramid over $P_{6,10,14}$ is a facet of $P$, then $P$ is a pyramid (over $P_{7,14,16}$).
Thus, $P$ has 17 facets and at least 15 of them are combinatorially equivalent to $P_{7,14,16}$.
Consequently, $P$ has at least 18 vertices ($P_{7,14,16}$ has 4 additional vertices w.r.t. $P_{6,10,14}$).
An incidence matrix of $P$ is shown on Fig.~\ref{fig:P8}.
\begin{figure}
	\centering
	\begin{tabular}{c|*{7}{c}|*{7}{c}|}
		\multicolumn{1}{c}{}
		& $v_1$ & \dots & $v_{10}$ & $v_{11}$ & $v_{12}$ & $v_{13}$ & \multicolumn{1}{c}{$v_{14}$} & $v_{15}$ & $v_{16}$ & $v_{17}$ & $v_{18}$ & $v_{19}$ & \dots & \multicolumn{1}{c}{$v_{n}$}\\
		\cline{2-15}
		$F_1$ &  &  &  &  &  &  &  & $\star$ &  &  & \dots &  &  & $\star$\\
		\vdots &  &  &  \multicolumn{3}{c}{$P_{7,14,16}$}  &  &  & \vdots &  &  & $\ddots$ &  &  & \vdots\\
		$F_{15}$ &  &  &  &  &  &  &  & $\star$ &  &  & \dots &  &  & $\star$\\
		$F_{16}$ & 1 & \dots & 1 & 0 & 0 & 0 & 0 & 1 & 1 & 1 & 1 & 0 & \dots & 0\\ 
		\cline{2-8}
		$F_{17}$ & 1 & \dots & 1 & 1 & 1 & 1 & 1 & 0 &   &  & \dots &  &  & 0\\ 
		\cline{2-15}
	\end{tabular}
	\caption{An incidence matrix of an 8-polytope $P$, one of whose facets is $P_{7,14,16}$}\label{fig:P8}
\end{figure}
An incidence matrix of $P_{7,14,16}$ occupies rows $F_1$--$F_{16}$ and columns $v_1$--$v_{14}$. 
Without loss of generality, we suppose that $v_{11}, \dots, v_{14} \notin F_{16}$ and $v_{15}, \dots, v_{18} \in F_{16}$. Hence, if we remove from the matrix columns $v_{11}$, \dots, $v_{14}$, $v_{19}$, \dots, $v_{n}$, and the row $F_{16}$, then we get a matrix equivalent to an incidence matrix of $P_{7,14,16}$ (up to permutations of columns and rows). By the computer program, we enumerate all the variants and have found only one (up to permutations of the last 4 columns). The values in columns $v_{15}$--$v_{18}$ coincides with values in columns $v_{11}$--$v_{14}$, except the values in rows $F_{16}$ and $F_{17}$.
The result is shown on Fig.~\ref{fig:P82}.
\begin{figure}
	\centering
	\begin{tabular}{c|*{14}{c}|*{7}{c}|}
		\multicolumn{1}{c}{} & $v_1$ & \multicolumn{12}{c}{\dots} & \multicolumn{1}{c}{$v_{14}$} & \multicolumn{3}{c}{\dots} & $v_{18}$ & \multicolumn{2}{c}{\dots} & \multicolumn{1}{c}{$v_{n}$}\\
		\cline{2-22}
$F_{1}$ & 0 & 0 & 0 & 0 & 0 & 0 & 1 & 1 & 1 & 1 & 1 & 1 & 1 & 0 & 1 & 1 & 1 & 0 & $\star$ & \dots & $\star$\\
\multirow{4}{*}{\vdots} & 1 & 1 & 1 & 1 & 1 & 1 & 0 & 0 & 0 & 0 & 0 & 0 & 0 & 1 & 0 & 0 & 0 & 1 &  \multirow{13}{*}{\vdots} & & \multirow{13}{*}{\vdots} \\
 & 0 & 1 & 1 & 1 & 1 & 0 & 1 & 1 & 1 & 1 & 1 & 1 & 0 & 0 & 1 & 1 & 0 & 0 &  &&\\ 
 & 1 & 0 & 0 & 1 & 1 & 1 & 1 & 1 & 1 & 1 & 1 & 0 & 1 & 0 & 1 & 0 & 1 & 0 &  &&\\
 & 1 & 1 & 1 & 0 & 0 & 1 & 1 & 1 & 1 & 1 & 0 & 1 & 1 & 0 & 0 & 1 & 1 & 0 &  &&\\
$F_{6}$ & 0 & 0 & 1 & 0 & 1 & 1 & 1 & 1 & 1 & 0 & 1 & 1 & 1 & 1 & 1 & 1 & 1 & 1 &  &&\\
\multirow{9}{*}{\vdots} & 0 & 1 & 0 & 1 & 0 & 1 & 1 & 1 & 0 & 1 & 1 & 1 & 1 & 1 & 1 & 1 & 1 & 1 &  &&\\
 & 1 & 0 & 1 & 1 & 0 & 0 & 1 & 0 & 1 & 1 & 1 & 1 & 1 & 1 & 1 & 1 & 1 & 1 &  &&\\
 & 1 & 1 & 0 & 0 & 1 & 0 & 0 & 1 & 1 & 1 & 1 & 1 & 1 & 1 & 1 & 1 & 1 & 1 &  &&\\
 & 0 & 1 & 1 & 1 & 1 & 1 & 1 & 1 & 0 & 0 & 1 & 1 & 0 & 1 & 1 & 1 & 0 & 1 &  &&\\
 & 1 & 1 & 1 & 1 & 1 & 0 & 0 & 0 & 1 & 1 & 1 & 1 & 0 & 1 & 1 & 1 & 0 & 1 &  &&\\
 & 1 & 0 & 1 & 1 & 1 & 1 & 1 & 0 & 1 & 0 & 1 & 0 & 1 & 1 & 1 & 0 & 1 & 1 &  &&\\
 & 1 & 1 & 0 & 1 & 1 & 1 & 0 & 1 & 0 & 1 & 1 & 0 & 1 & 1 & 1 & 0 & 1 & 1 &  &&\\
 & 1 & 1 & 1 & 0 & 1 & 1 & 0 & 1 & 1 & 0 & 0 & 1 & 1 & 1 & 0 & 1 & 1 & 1 &  &&\\
 & 1 & 1 & 1 & 1 & 0 & 1 & 1 & 0 & 0 & 1 & 0 & 1 & 1 & 1 & 0 & 1 & 1 & 1 & $\star$ & \dots & $\star$\\
$F_{16}$ & 1 & 1 & 1 & 1 & 1 & 1 & 1 & 1 & 1 & 1 & 0 & 0 & 0 & 0 & 1 & 1 & 1 & 1 & 0 &\dots&0\\
$F_{17}$ & 1 & 1 & 1 & 1 & 1 & 1 & 1 & 1 & 1 & 1 & 1 & 1 & 1 & 1 & 0 & 0 & 0 & 0 & 0 &\dots&0\\
		\cline{2-22}
\end{tabular}
\caption{An incidence matrix of an 8-polytope $P$, one of whose facets is $P_{7,14,16}$}\label{fig:P82}
\end{figure}
Every of rows $F_6$--$F_9$ has exactly 14 ones. 
Hence, we can extract from them incidence matrices for $P_{7,14,16}$.
It is easy to check, that the facet $F_6$ has 14 vertices and 15 facets (after removing 4 columns with zeroes, the row $F_2$ will be a subset of $F_{10}$). So, it cann't be $P_{7,14,16}$.
We get a contradiction.
There are only one 2-neighborly 8-polytope with at most 17 facets and a facet $P_{7,14,16}$.
It is a pyramid over $P_{7,14,16}$.

Suppose, a 2-neighborly 10-polytope $P$ has at most 19 facets, and one of facets is combinatorially equivalent to $P_{9,12,18}$. We does not consider the case, when $P$ is a pyramid. Hence, $P$ has no a facet with 13 or more vertices. Thus, every row in an incidence matrix of $P$ has at most 12 ones. By Corollary~\ref{cor:1}, every column in an incidence matrix of $P$ has at least 13 ones. Since every row in an incidence matrix of $P_{9,12,18}$ has exactly 10 ones, then $P$ has at most 14 vertices.
By our program, we enumearte all such matrices and found only one suitable~--- a pyramid over $P_{9,12,18}$.





%
%

\end{document}